\documentclass{amsart}
\usepackage{graphicx} 

\usepackage{amsfonts}
\usepackage{amssymb}
\usepackage{amsmath}
\usepackage{amsthm}
\usepackage{mathtools}
\usepackage{enumerate, enumitem}
\usepackage{mathrsfs}
\usepackage{xcolor}
\usepackage{hyperref}

\newtheorem*{theorem*}{Theorem}
\newtheorem*{question*}{Question}

\newtheorem{theorem}{Theorem}[section]

\newtheorem*{proposition*}{Proposition}
\newtheorem{lemma}[theorem]{Lemma}
\newtheorem*{lemma*}{Lemma}
\newtheorem{corollary}[theorem]{Corollary}
\newtheorem*{corollary*}{Corollary}
\newtheorem{fact}[theorem]{Fact}
\newtheorem*{fact*}{Fact}

\theoremstyle{definition}
\newtheorem{definition}[theorem]{Definition}
\newtheorem*{definition*}{Definition}
\newtheorem{claim}[theorem]{Claim}
\newtheorem*{claim*}{Claim}

\newtheorem*{conjecture*}{Conjecture}

\theoremstyle{definition}

\theoremstyle{remark}

\newtheorem*{example*}{Example}

\newtheorem*{remark*}{Remark}

\newtheorem*{note*}{Note}

\newcommand{\Z}{\ensuremath{\mathbb{Z}}}

\newcommand{\frc}{\ensuremath{\mathfrak{c}}}

\newcommand{\A}{\ensuremath{\mathbf{A}}}

\newcommand{\cB}{\mathcal{B}}

\newcommand{\ZFC}{\mathrm{ZFC}}

\newcommand{\defeq}{\vcentcolon=}

\newcommand{\acc}{\mathrm{acc}}
\newcommand{\dom}{\mathrm{dom}}

\newcommand{\mcI}{\mathcal I}
\newcommand{\cf}{\mathrm{cf}}
\newcommand{\ra}{\rightarrow}

\AtBeginDocument{%
   \def\MR#1{}
}
\begin{document}

\title{Nonvanishing derived limits from $\clubsuit$-type principles}

\author[M. Casarosa]{Matteo Casarosa}
\address{Departament de Matemàtiques i Informàtica, Universitat de Barcelona, Gran Via de
les Corts Catalanes 585, 08007 Barcelona, Catalonia}
\email{matteo.casarosa@ub.edu}
\urladdr{https://sites.google.com/view/matteocasarosa/}

\subjclass[2020]{03E05, 03E17, 03E75, 18G10}
\keywords{derived limits, pro-category of abelian groups, nontrivial coherence, guessing principles}
\thanks{I would like to thank Jeffrey Bergfalk, Constance Bromham, Andrew Brooke-Taylor, Chris Lambie-Hanson, and Assaf Rinot for helpful conversations.}

\begin{abstract}
Combinatorial set theory provides several tools to study derived limits of certain inverse systems of abelian groups. Most known nonvanishing results for $\lim^n$ with $n>1$ depend on some guessing principles of the form $\mathrm{w}\lozenge(S)$ called \emph{weak diamonds}. In this paper, we explore some applications of the guessing principle $\clubsuit(S)$ and its weakenings instead.    

\end{abstract}

\maketitle

\section{Introduction}

In recent years, there has been significant progress in applying set theory to the derived functors $\lim^n$ of the inverse limit, which measure failures of the inverse limit functor to be exact.  Already in \cite{Gob70}, Goblot showed that inverse systems of low enough cofinality have vanishing derived limits for a sufficiently large $n$. Later, Mitchell (\cite{mitchell1972rings}) showed the optimality of Goblot's hypotheses by producing some (non-surjective) inverse systems $\mathbf{G}_n$ of cofinality $\aleph_n$ such that $\ZFC \vdash \lim^{n+1} \mathbf{G}_n \neq 0$. 

Mardesic and Prasolov in \cite{mardevsic1988strong} indicated that the vanishing of the derived limits of the inverse system $\mathbf{A}$ indexed over $({}^\omega \omega, \leq)$ with objects $A_f = \bigoplus_{i< \omega} \mathbb{Z}^{f(i)}$ and the natural projections as bonding morphisms is a necessary condition for the additivity of strong homology. This system has therefore received much attention, together with the more general systems indexed over $(\mathcal{I},\subseteq)$ where $\mathcal{I}$ is a collection of sets and for any $a \in \mathcal{I}$ the corresponding object is $\bigoplus_{x \in a} H$, for some nontrivial abelian group $H$. Mardesic and Prasolov already noticed that under $\mathrm{CH}$ we have $\lim^1 \mathbf{A} \neq 0$. This turns out to be a consequence of the fact that vanishing phenomena are invariant under taking a cofinal subsystem, together with a very general nonvanishing result for systems indexed over $\subseteq^*$-chains of length $\omega_1$ (see \cite{bekkali2006topics}, Proposition 4.22). Later, Prasolov (\cite[Theorem 3]{PRASOLOV2005493}) showed that not only is this group nonvanishing, but its size is maximal.  

More recently, some results have been obtained for \emph{higher derived limits} (i.e., $\lim^n$ for $n>1$). The proofs are based on some ``step-up'' arguments that, starting from the $\ZFC$ base case of a nonvanishing $\lim^1$, use witnesses for nonvanishing derived limits on sets of cofinality $\aleph_n$ to build one for a set of cofinality $\aleph_{n+1}$. These arguments invariably involve some guessing principles, unless $H$ from the definition above is some large direct sum of groups, similarly to Mitchell's example. In particular:

\begin{enumerate}

    \item Bergfalk (\cite{berg}) derived $\lim^2 \mathbf{A} \neq 0$ from $\mathfrak{b}=\mathfrak{d}=\mathfrak{c}=\aleph_2 + \lozenge(S_1^2)$, which is a consequence of $\mathrm{PFA}$. The guessing principle $\lozenge(S_1^2)$ implies $2^{\aleph_0} \leq \aleph_2$ and this, together with the theorem of Goblot, was the main obstruction to extending the result to $n>2$.

    \item Velickovic and Vignati (\cite{VV24}) proved that, for every $n \in \mathbb{N}$, it is consistent that $\lim^n \mathbf{A} \neq0$. The proof assumes $\bigwedge_{k < n} \mathrm{w}\lozenge (S_k^{k+1})$, where $\mathrm{w}\lozenge (S)$ for a stationary set $S \subseteq \kappa^+$ denotes the corresponding \emph{weak diamond} principle. This is compatible with a large continuum, but still requires some cardinal arithmetic. In particular, $\mathrm{w}\lozenge(\kappa^+)$ is equivalent to $2^\kappa < 2^{\kappa^+}$.

    \item The author and Lambie-Hanson \cite{CASAROSA2026111076} proved that, in fact, $\mathfrak{d}= \aleph_n + \bigwedge_{k < n} \mathrm{w}\lozenge (S_k^{k+1})$ already implies $\lim^n \mathbf{A} \neq0$. The weak diamonds can be dropped for the analogous system $\mathbf{A}[H]$ where $\mathbb{Z}$ is replaced by $\bigoplus_{\alpha \in \omega_n} \mathbb{Z}$.

\end{enumerate}

As the progression above shows, even though there has been progress on the front of the guessing principles necessary for the step-up argument, these have so far remained tied to a very specific kind of guessing and the consequent cardinal arithmetic. 

In this paper, we explore the nonvanishing consequences of the different, orthogonal guessing principle $\clubsuit(S)$ and of some of its weakenings. In particular, we show that $\mathfrak{d}=\mathfrak{c}=\aleph_2$ plus $\clubsuit(S_1^2)$ implies $\lim^2 \mathbf{A} \neq 0$. We also prove a similar theorem for $\lim^2$ for $\subseteq^*$-chain of length $\omega_2$ or higher. Both results adapt some ideas from \cite{CASAROSA2026111076}. 

\section{Preliminaries}

\begin{definition}
  Suppose that $\mcI$ is a collection of sets, $n$ is a positive integer, and 
  $H$ is an abelian group. An \emph{$H$-valued $n$-family indexed over 
  $\mcI$} is a family of functions of the form 
  \[
    \Phi = \left \langle \varphi_u : \bigcap_{i < n} u(i) \ra H ~ \middle| ~ u \in 
    \mcI^n \right \rangle.
  \]
An $n$-family is said to be
  \begin{itemize}
    \item \emph{alternating} if, for every $u \in \mcI^n$ and every permutation 
    $\sigma:n \rightarrow  n$, we have 
    \[
      \varphi_{u \circ \sigma} = \mathrm{sgn}(\sigma) \varphi_u;
    \]
    \item \emph{coherent} if it is alternating and, for all $v \in \mcI^{n+1}$, 
    we have 
    \[
      \sum_{i < n+1} (-1)^i \varphi_{v^i} =^* 0;
    \]
    \item \emph{trivial} if 
    \begin{itemize}
      \item $n = 1$ and there is a function $\psi : \bigcup \mcI \ra H$ 
      such that, for all $u \in \mcI$, we have $\psi \restriction u 
      =^* \varphi_u$; or
      \item $n > 1$ and there is an alternating $(n-1)$-family 
      \[
        \Psi = \left\langle \psi_t : \bigcap_{i < n-1} t(i) \ra H \ \middle| \ t \in 
        \mcI^{n-1} \right\rangle
      \] 
      such that, for all $u \in \mcI^n$, we have
      \[
        \varphi_u =^* \sum_{i < n} (-1)^i \psi_{u^i}.
      \]
    \end{itemize}
    In this case, we say that $\psi$ or $\Psi$ \emph{trivializes} $\Phi$.
  \end{itemize}
\end{definition}

The fact that every trivial family is coherent corresponds to a foundational observation from homological algebra. This is the fact that $\partial \circ \partial =0$ where $\partial$ is the coboundary operator given by the alternating sum above. It is also easy to see that $n$-coherent families modulo $n$-trivial ones form a group:

\begin{definition}
For a family of sets $\mathcal{I} \subseteq \mathcal{P}(X)$ and an abelian group $H$, we write   $\mathbf{A}_\mathcal{I}[H]$ for the inverse system of abelian groups given by $A_a = \bigoplus_{x \in a} H$. We omit the subscript when $\mcI= \{ I_f  \mid f \in {}^\omega \omega   \}$, where $I_f = \{ (i, j) \mid j \leq f(i) \} $, and omit the group when $H = \mathbb{Z}$. 
\end{definition}

\begin{fact}

The derived limit $\lim^n \mathbf{A}_\mathcal{I}[H]$ is isomorphic to the group of $n$-coherent $H$-valued families indexed over $\mcI$ modulo the $n$-trivial ones.
    
\end{fact}

Most of the set-theoretic research on the derived limits $\lim^n$ has focused on whether they vanish or not. For many choices of indexing family $\mcI$ and codomain group $H$, this question is independent of $\mathrm{ZFC}$. However, in many cases, the following combinatorial adaptation of a theorem of Goblot \cite{Gob70} gives a positive answer.

\begin{lemma}[\cite{CASAROSA2026111076}[Proposition 2.7]

\label{lemma_goblot}
Suppose that $H$ is an abelian group, $0<n<\omega$, $\mcI$ is a collection of sets with $\cf(\mcI, \subseteq^*) < \aleph_n$, and $\Phi= \langle \Phi_u \mid u \in \mcI^n \rangle$ is an $H$-valued coherent $n$-family. Then $\Phi$ is trivial.

\end{lemma}

Moreover, it is easy to see that an $n$-coherent family on a $\subseteq^*$-cofinal set extends to the entire indexing set. The proof of \cite[Proposition 2.4]{CASAROSA2026111076} effectively shows that any trivialization of a given family extends similarly.   

\begin{lemma}[after \cite{CASAROSA2026111076}[Proposition 2.4]
\label{lemma_extendtriv}
Suppose $\Psi$ is a trivialization of $\Phi \restriction \mathcal{J}$ and $\mathcal{J}$ is a $\subseteq^*$-cofinal subset of $\mcI$, where $\Phi$ is indexed over $\mcI$. Then there exists $\Psi'$ trivializing $\Phi$ such that $\Psi' \restriction \mathcal{J} = \Psi$.

\end{lemma}

Even though set-theoretic research on derived limits has focused on the question of vanishing, some authors have also given cardinality estimates for the group of $1$-coherent nontrivial $\mathbb{Z}/2\mathbb{Z}$-valued families over an $\subseteq^*$-increasing chain of countable sets of length (or cofinality) $\omega_1$. The lower bound $2^{\aleph_0}$ was proved within the framework of the \emph{gap cohomology group} by Talayco (\cite{TALAYCO199569}[Corollary 18]). Farah proved within the same framework the optimal estimate of $2^{\aleph_1}$ in \cite{farah1996coherent}. It is easy to see that the $\mathbb{Z}/2\mathbb{Z}$-valued case serves as a lower bound for the estimate of the size of $\lim^1$ with different codomains: in fact, given a family of coherent $\Z/2\Z$-valued families whose pairwise difference is nontrivial, fix $1_H \in H \setminus \{ 0_H \}$ and consider the corresponding $H$-valued families with $0_H$ replacing $0$ and $1_H$ replacing $1$; if the difference $\Phi_H -\Phi'_H  $ of a pair of these was trivialized by some $\{ 1_H, 0, - 1_H\}$-valued function $\psi$ then its ``absolute value'' would correspond to a trivialization of the difference $\Phi-\Phi'$ of the corresponding $\Z/2\Z$-valued pair, a contradiction. This proves a lower bound for every group, which is clearly an exact estimate e.g. for countable groups. More generally, Prasolov effectively proved in \cite{PRASOLOV2005493} the following:   

\begin{lemma}[after (\cite{farah1996coherent} and \cite{PRASOLOV2005493}[Theorem 3] 
\label{cardinality_lemma}
Let $\mathcal{C}_{\omega_1}$ be an increasing chain of cofinality $\omega_1$ in $\mathcal{P}(\omega)/\mathrm{fin}$. Then for any nontrivial abelian group $H$ we have $ \vert \lim^1 \mathbf{A}_{\mathcal{C}_{\omega_1}} [H]\vert = \vert H \vert^{\aleph_1}$.

\end{lemma}

We now introduce a series of progressively weaker guessing principles whose bearing on nonvanishing derived limits is the focus of this paper. 

Throughout the following, let $\lambda$ denote an uncountable regular cardinal, and $S$ denote a stationary subset of $\lambda$.

\begin{definition}
The principle $\clubsuit(S)$ asserts the existence of a sequence $\langle A_\alpha \mid \alpha \in S \rangle$
such that:

\begin{itemize}
    \item for all $\alpha \in S$, $A_\alpha$ is a cofinal subset of $\alpha$;
    \item if $Z$ is a cofinal subset of $\lambda$, then the following set is stationary:
    \[
    \{\alpha \in S \mid A_\alpha \subseteq Z\}.
    \]
\end{itemize}
\end{definition}

\begin{definition}
Let $\overline{\mu}=\langle \mu_\alpha \mid \alpha \in \lambda \rangle  $  be a sequence of cardinals. The principle $\clubsuit^{\leq\overline \mu}(S)$ asserts the existence of a sequence $\langle \cB_\alpha \mid \alpha \in S \rangle$ such that:

\begin{itemize}
    
    \item If $B \in \cB_\alpha$, then $B$ is an unbounded subset of $\alpha$.  

    \item $\vert \cB_\alpha \vert \leq \mu_\alpha$. 
    
    \item For every $Z \in [\lambda]^\lambda$, 

\[      \{ \alpha \in S \mid \exists B \in \cB_\alpha \,\, B \subseteq Z   \}   \]

is stationary

\end{itemize}

We simply write $\clubsuit^\leq(S)$ if $\mu_\alpha = \vert \alpha \vert $, and write $\clubsuit^{<\mu}(S)$ if $\mu_\alpha < \mu$ for every $\alpha$.  

\end{definition}

Rather than asking for inclusion, some guessing principles only ask that a guess meet the object to be guessed unboundedly in $\alpha$. We are particularly interested in the following guessing with partial functions.

\begin{definition}
$\clubsuit^{\leq \overline{\mu}}_\rightharpoonup (S)$ asserts the existence of a sequence $\langle \mathcal F_\alpha \mid \alpha \in S \rangle$ such that:

\begin{itemize}
    \item If $g \in \mathcal{F}_\alpha$ then $g$ is a partial function from $\alpha$ to $\alpha$ with domain unbounded in $\alpha$.

    \item $\vert \mathcal{F}_\alpha \vert \leq \mu_\alpha$.

    \item For every function $f: \lambda \to \lambda$, the set 

    \[ \{ \alpha \in S \mid   \exists g \in \mathcal{F}_\alpha \,\,  \sup\{\beta \in \dom(g) \mid g(\beta) = f(\beta) \} =\alpha \}   \]

    is stationary.

\end{itemize}

Again, we drop $\overline{\mu}$ from the notation when $\mu_\alpha = \vert \alpha \vert$, and write ``$<\mu$'' when the number of guesses is strictly below a fixed cardinal $\mu$.
\end{definition} 

$\clubsuit^\leq_\rightharpoonup (S)$ is clearly implied by and very similar to the principle considered by Dzamonja and Shelah in \cite[Question 2.6]{dvzamonja1996saturated}. There, Dzamonja and Shelah ask whether this principle follows from what they call $\clubsuit^*_-(S)$, which is a theorem of $\ZFC$ whenever $\lambda= \mu^+$, $\cf(\mu) = \kappa$ and $S \subseteq S^\lambda_{\neq \kappa}$ is a stationary set of points of cofinality different from $\kappa$. The question appears to be open. A positive answer would give the first examples of nonvanishing results for higher derived limits without guessing principles beyond $\ZFC$  for families taking values in small groups (e.g., $\Z$) rather than large direct sums. Now we show that, at the very least, $\clubsuit^{\leq \overline{\mu}}(S)$ implies $\clubsuit^{\leq \overline{\mu}}_\rightharpoonup(S)$.

\begin{theorem}
For any sequence of cardinals $\overline{\mu}$,   $\clubsuit^{\leq \overline{\mu}}(S) \Longrightarrow  \clubsuit^{\leq \overline{\mu}}_\rightharpoonup(S)$.
\end{theorem}

\begin{proof}
Let $ \langle \cB_\alpha\mid\alpha\in S\rangle$ be a $\clubsuit^{\leq \overline{\mu}}(S)$-sequence, and for each $\alpha$ let $\cB_\alpha = \langle B^i_\alpha \mid i < \mu_\alpha \rangle$. 

Fix a bijection $ \varphi:\lambda\times\lambda\rightarrow\lambda$ and let $C= \lbrace\alpha \in \acc(\lambda) \mid  \varphi \restriction(\alpha\times\alpha):
\alpha\times\alpha\rightarrow\alpha \text{ is a bijection}\}$, which is a club. Now for $\alpha \in S$, and $i< \mu_\alpha$ let $g^i_\alpha = \varphi^{-1} [B^i_\alpha]$ if this is a partial function from $\alpha$ to $\alpha$ with domain unbounded in $\alpha$. Otherwise, let $g^i_\alpha$ be any such partial function. Finally, let $\mathcal{G}_\alpha = \langle g^i_\alpha \mid i < \mu_\alpha \rangle$. 

Let $f: \lambda \rightarrow \lambda$ be the function we want to guess. Let $D$ be the club of its closure points. Note that $\vert f \vert = \lambda$ and so $\varphi[f]$ is cofinal in $\lambda$. Now apply $\clubsuit^{\leq \overline{\mu}}(S)$ to $\varphi[f]$ to obtain a stationary set $S' \subseteq S$ such that, for every $\alpha \in S'$ and for some $i_\alpha < \mu_\alpha$, $B^{i_\alpha}_\alpha \subseteq \varphi[f]$. 

We prove that $\langle \mathcal G_\alpha \mid \alpha \in S \rangle$ is a $\clubsuit^{\leq \overline{\mu}}_\rightharpoonup(S)$-sequence and guesses correctly on the stationary set $S'' =S' \cap \acc(C \cap D)$.  

Indeed, for any such $\alpha$, we have  $\varphi^{-1}[B^{i_\alpha}_\alpha]\subseteq f$, hence it is a partial function (since it is included in a function) and is from $\alpha$ to $\alpha$ (since $B^{i_\alpha}_\alpha \subseteq \alpha$ and $\alpha \in C$). Now we prove that $\sup(\dom(\varphi^{-1}[B^{i_\alpha}_\alpha]))=\alpha$. Let $\pi_1$ and $\pi_2$ be the projections on the domain and the codomain coordinates, respectively. Suppose towards a contradiction that 
$(\pi_1\circ \varphi^{-1})[B^{i_\alpha}_\alpha] \subseteq \gamma < \alpha$. Then, since $\alpha \in \acc(C \cap D)$, there exists $\beta \in C \cap D$ such that $\gamma < \beta < \alpha$. Now, since $\beta \in D$, we have $(\pi_2\circ\varphi^{-1})[B^{i_\alpha}_\alpha]\subseteq\beta$, so $\varphi^{-1}[B^{i_\alpha}_\alpha]\subseteq\beta\times\beta$ and since $\beta \in C$ we have $(\varphi \circ \varphi^{-1}) [B^{i_\alpha}_\alpha] = B^{i_\alpha}_\alpha \subseteq \beta < \alpha$ a contradiction to the fact that $B^{i_\alpha}_\alpha$ is cofinal in $\alpha$. So, $ g^{i_\alpha}_\alpha = \varphi^{-1}[B^{i_\alpha}_\alpha] \subseteq f $ and, since $\sup(\dom(g^{i_\alpha}_\alpha))=\alpha$, we have indeed  $\sup\{\beta \in \dom(g^{i_\alpha}_\alpha) \mid g^{i_\alpha}_\alpha(\beta) = f(\beta) \}  =\alpha$.  
\end{proof}

Finally, we define some more notation. The notation $\acc(X)$ stands for the set of accumulation points of a set $X$ and by ``$\Phi\restriction \restriction a$'' we mean the restriction of all the functions in the family $\Phi$ to the domain $a$, and by  $\Phi {}^\frown \langle0\rangle$'' we mean the end-extension of a family $\Phi$ by the constant $0$ function.

\section{The case $\mathfrak{d}=\mathfrak{c}=\aleph_2$}

The following theorem brings us closer to an affirmative answer to \cite{WITHOUTLC}[Question 7.3], asking whether $2^{\aleph_0} \leq \aleph_2$ implies either $\lim^1 \mathbf{A} \neq 0$ or $\lim^2 \mathbf{A} \neq 0$. 

\begin{theorem}
$\mathfrak{d}=\mathfrak{c}=\aleph_2+\clubsuit(S_1^2) \Rightarrow \lim^2 \mathbf{A} \neq 0$.

\end{theorem}

\begin{proof}
We introduce a construction that uses some of the properties of the notion of \emph{ascending sequence} from \cite{CASAROSA2026111076}.

\begin{claim}
If $\mathfrak{d}= \aleph_2$, then there exists a sequence of functions $\mathcal{F} = \langle f_\alpha \mid \alpha< \omega_2 \rangle$ cofinal in $({}^\omega \omega, \leq^*)$ and such that:

\begin{enumerate}
    \item For every $\alpha < \gamma$, $f_\gamma \not \leq^*  f_\alpha$.

    \item For every $\alpha < \beta < \gamma$, letting $e_{\alpha, \beta} = I_{f_\beta} \setminus I_{f_\alpha}$, the set  $I_{f_\gamma} \cap e_{\alpha, \beta}$ is infinite. 
\end{enumerate}
\end{claim}    

\begin{proof}[Proof of Claim]

We construct $\mathcal{F}$ recursively. For $\gamma < \omega_2$, we pick the $\gamma$-th element $f_{\gamma, 0}$ of some cofinal set. Then we observe that $\vert \mathcal{F} \restriction \gamma \vert<\aleph_2$. So, we can find $f_{\gamma, 1}$ not dominated by any of these functions. Assuming recursively that condition $(1)$ is satisfied below $\gamma$, for every $\alpha < \beta < \gamma$, we have $\vert e_{\alpha, \beta} \vert = \aleph_0$. Moreover, since $\vert \langle e_{\alpha, \beta} \mid \alpha < \beta < \gamma \rangle \vert < \aleph_2$, we can consider the projection to the first coordinate $\pi$ and the step functions $h_{\alpha, \beta}$ taking value $f_\alpha(\pi(e_{\alpha, \beta})(n))$ on the interval $(\pi (e_{\alpha, \beta}(n-1)), \pi(e_{\alpha, \beta}(n))]$ (we are using function notation to denote the $n$-th element in the increasing enumeration of a set). Then we find a single weakly increasing $f_{\gamma, 2} \not \leq^* h_{\alpha, \beta}$ for all $\alpha < \beta < \gamma$. Since it is weakly increasing, it will contain the right endpoint $e_{\alpha, \beta}(n)$ of the infinitely many intervals where it is not pointwise dominated by $h_{\alpha, \beta}$. Finally, let $f_\gamma = f_{\gamma, 0} \vee f_{\gamma, 1} \vee f_{\gamma, 2} $. This is cofinal and satisfies conditions $(1)$ and $(2)$. \end{proof}

Observe that, for a club $D \subseteq \omega_2$, $\gamma \in D$ implies that $\mathcal{F}\restriction\gamma$ is $\leq^*$-directed.  

Note that, since $\frc = \aleph_2$, we can take each  $\alpha < \omega_2$ to code a different element of ${}^{\omega \times \omega} \mathbb{Z}$, so that, once these functions are restricted to the various $I_{f_\xi}$, whenever $G_\alpha$ is a partial function from $\alpha$ to $\alpha$ it can be read off as a $1$-dimensional family of functions $\langle \psi_\xi: I_{f_\xi} \rightarrow \mathbb{Z} \mid \xi \in \dom(G_\alpha) \rangle$.

We build our coherent nontrivial $2$-family $\Phi$ on $\mathcal{F}$ by recursion. Suppose $\Phi \restriction (\mathcal{F}\restriction \gamma)$ has already been defined. If not all of the following conditions hold, we simply trivialize $\Phi \restriction (\mathcal{F} \restriction \gamma)$ using \cite{CASAROSA2026111076}[Proposition 2.7] and extend arbitrarily to a coherent $\Phi \restriction (\mathcal{F} \restriction\gamma+1)$.

If, however, $\gamma \in S_1^2 \cap D$, $G_\gamma$ is a partial function such that $\sup (\dom(G_\gamma))=\gamma$, the family coded by $G_\gamma$ trivializes $\Phi \restriction\dom(G_\gamma)$, and this trivialization extends to a trivialization $\Upsilon^\gamma$ of  $\Phi \restriction(\mathcal{F}\restriction\gamma)$, then we ensure that $G_\gamma$ does not extend to a trivialization of $\Phi \restriction ( \mathcal{F} \restriction \gamma+1)$. To this end, we want to define a coherent $1$-family $\Theta^\gamma$ over $\mathcal{F}\restriction\gamma$, such that $(\Theta^\gamma \restriction \dom(G_\gamma))\restriction \restriction I_{f_\gamma}$ is nontrivial. 

For each function $s \in {}^{\omega_1} 2$ we produce a coherent family $\Theta^{\gamma, s}$ on $\mathcal{F}\restriction \gamma$ as follows.  
Consider a countable set $(\delta_ n)_{n< \omega} \subseteq \gamma$, and let $\hat{f}_{\delta_{ n}}= \bigvee_{i=0}^n f_{\delta_{i}}$. Note that since $\gamma \in D$, for each of these there is a corresponding $\alpha_{n} < \gamma$ such that  $\hat{f}_{\delta_{n}} \leq^* f_{\alpha_{n}}$. Now take a $\clubsuit(S_1^2)$-sequence, which yields a sequence $\langle G_\alpha \mid \alpha \in S_1^2 \rangle$ of partial functions with cofinal domain in $\alpha$. Then we can find $\alpha =\sup_n \alpha_{n} < \beta \in \dom(G_\gamma)$, and for this $I_{f_\gamma} \cap e_{\alpha_n, \beta}$ is infinite. We can therefore pick $x_{n} \in I_{f_\gamma} \cap I_{f_{\beta}} \setminus I_{\hat{f}_{\delta_n}}$, so that $X = \{ x_{n}   \}_{n < \omega} \subseteq I_{f_\gamma} \cap I_{f_{\beta}}$ is almost disjoint from every $I_{f_{\delta_n}}$.  

Now let $\iota_\gamma : \gamma \to \omega_1 $ be an injection (and a bijection whenever $\vert \gamma \vert = \aleph_1$). Assume we have already defined a coherent $\Theta^{\gamma, s} \restriction \iota_\gamma^{-1} [\delta] \cup \{ \beta_{\gamma, \xi}  \}_{\xi< \delta} $ for some $\delta< \omega_1$. Since this family is countable, applying \cite{CASAROSA2026111076}[Proposition 2.7] we find a trivialization $\psi: \omega\times \omega \rightarrow \mathbb{Z}$. Then we find $\beta_{\gamma, \delta} \in \dom(G_\gamma)$ as above where the countable set $(\delta_{n})_{n< \omega}$ is an enumeration of $\iota_\gamma^{-1} [\delta] \cup \{ \beta_{\gamma, \xi}  \}_{\xi< \delta}$. Then one uses the trivialization $\psi$ to define a coherent extension $\Theta'_{\gamma, s\restriction\delta}$ of our family to $\iota_\gamma^{-1}[\delta+1] \cup \{\beta_{\gamma, \xi}  \}_{\xi \leq \delta} $. Finally, we define $\Theta^{\gamma, s} \restriction\iota_\gamma^{-1}[\delta+1] \cup \{\beta_{\gamma, \xi}  \}_{\xi \leq \delta}  $ by taking  $\Theta'_{\gamma, s\restriction\delta}$ and changing the function with domain $I_{f_{\beta_{\gamma, \delta}}}$ to take constant value $s(\delta)$ over the set $X_{\gamma, \delta} \subseteq I_{f_{\beta_{\gamma, \delta}}} \cap I_{f_\gamma}$ defined as the $X$ from above. This preserves coherence since $X_{\gamma, \delta}$ is almost disjoint from all the other domains. 

Now observe that for any two different $s, s' \in {}^{\omega_1}2$, the same $\psi_\gamma: I_{f_\gamma} \rightarrow \mathbb{Z}$ cannot trivialize both $\Theta^{\gamma, s}$ and $\Theta^{\gamma, s'}$. This is because if $s(\delta) \neq s'(\delta)$ then their functions with domain $I_{f_{\beta_{\gamma, \delta}}}$ will differ among themselves on every point of the infinite set $X_{\gamma, \delta} \subseteq I_{f_\gamma}$, so that no $\psi_\gamma$ can coincide with both modulo finite. 

Now observe that since $2^{\aleph_1} = \aleph_2+\clubsuit(S_1^2) \Leftrightarrow \lozenge(S_1^2)$, we can assume $2^{\aleph_1} > \aleph_2 = 2^{\aleph_0}$, otherwise $\lozenge(S_1^2)$ holds and we fall back into the hypotheses of \cite{BANNISTER_2025}[Theorem 3.1].

But then there must be an $s \in {}^{\omega_1}2$ such that $\Theta^{\gamma, s}\restriction (\dom(G_\gamma))\restriction \restriction I_{f_\gamma} $ is nontrivial, because there are only $2^{\aleph_0}$-many possibilities for $\psi_\gamma$ and $2^{\aleph_1}$-many families constructed as above.  

For such an $s$ we let $\Theta^\gamma = \Theta^{\gamma, s}$ and $\Phi\restriction (\mathcal{C \restriction}{\gamma+1}) = \partial((\Upsilon^\gamma+\Theta^\gamma) ^\frown\langle0\rangle)$.

We now show that $G_\gamma$ does not extend to a trivialization $\Upsilon^*$ of $ \Phi \restriction \gamma+1$. If it did, then for $\nu \in \dom(G_\gamma)$ we have 

\[ G_\gamma(\nu) -\Upsilon^*_{f_\gamma}  =^* \Phi_{f_\gamma, f_\nu}  =^* G_\gamma(\nu) + \Theta^\gamma_{f_\nu}
\]

on the common domain $I_{f_\nu} \cap I_{f_\gamma}$. But then, subtracting $G_\gamma(\nu)$ from both sides, we get $ -\Upsilon^*_{f_\gamma} \restriction I_{f_\nu} \cap I_{f_\gamma}  =^* \Theta_{f_\nu} \restriction I_{f_\nu} \cap I_{f_\gamma} $, but then $-\Upsilon^*_{f_\gamma}$ would trivialize $(\Theta \restriction \dom(G_\gamma)) \restriction \restriction I_{f_\gamma} $, a contradiction.

Finally, code any putative trivialization $\Psi$ of $\Phi$ as some $F: \omega_2 \rightarrow \omega_2$. Then for stationarily many  $\gamma \in S_1^2 \cap D$, we have $G_\gamma \subseteq F$. Now, if $G_\gamma$ did not extend to a trivialization of $\Phi \restriction \gamma$ that would be a contradiction, but even if it does, it cannot further extend to a trivialization of $\Phi\restriction (\mathcal{F} \restriction \gamma+1)$, again a contradiction. 
\end{proof}

\section{Long chains and squares}
\label{sect5}

\begin{definition}

Suppose that $\lambda$ is a regular uncountable cardinal and $S \subseteq \lambda$ is stationary. The principle $\square(\lambda, S)$ asserts the existence of a sequence $\vec{C} = \langle C_\alpha \mid \alpha < \lambda \rangle$ such that: 

\begin{enumerate}

\item for every limit ordinal $\alpha <\lambda$, $C_\alpha$ is a club in $\alpha$;

\item for all limit ordinals $\alpha < \beta < \lambda$, if $\alpha \in \acc(C_\beta)$, then $C_\beta \cap \alpha = C_\alpha$;

\item for every limit ordinal $\alpha < \lambda$, $C_\alpha \cap S = \emptyset$.

\end{enumerate}
Such a sequence is called a $\square(\lambda, S)$-sequence.
\end{definition}

The following generalizes and adapts a construction that appears in \cite{CASAROSA2026111076}[Lemma 5.3] in reference to chains in $({}^\omega \omega, \leq^*)$, as a way to obtain simultaneously nonvanishing derived limits. The point is ensuring coherence through the principle $\square(\lambda, S)$ at steps where Goblot's theorem is not applicable. 

\begin{theorem}
Suppose that $\mathfrak{c}=\lambda > \omega_1$ for some regular uncountable $\lambda$, and $\clubsuit^{<2^{\aleph_1}}_\rightharpoonup(S)$ and $\square(\lambda,S)$  both hold for some stationary set $S \subseteq S^\lambda_{\omega_1}$.
Then, for every strictly increasing chain $\mathcal{C} = \langle a_\delta \mid \delta \in \lambda \rangle$ in $\mathcal{P}(\omega)/\mathrm{fin}$, we have
\[
\lim\nolimits^2 \A_{\mathcal C}\neq0.
\]

\end{theorem}

\begin{proof}
We want to recursively build a $2$-coherent family $\Phi$ on $\mathcal{C}$. To simplify notation, we will write ``$\restriction \delta"$ rather than ``$\restriction (\mathcal{C}\restriction \delta)$''.

Let $\langle \mathcal{F}_\alpha \mid \alpha \in S \rangle$ be a $\clubsuit^{<2^{\aleph_1}}_\rightharpoonup(S)$-sequence.
Let moreover $\langle C_\delta \mid \delta < \lambda \rangle$ be a $\square(\lambda, S)$-sequence. We will use the latter to recursively preserve the following condition: 

\[  (\dagger) \,\,\, \forall \gamma\in \acc(\delta+1) \setminus S \,\, \forall \beta \in \acc(C_\gamma) \forall \alpha< \beta [ \phi_{\alpha, \beta} =^* \phi_{\alpha, \gamma}     ]                \]


Suppose we have defined $\Phi \restriction \delta$ and we want to define $\Phi \restriction\delta+1$. We will consider five cases:

\begin{itemize}
    \item \textbf{Case 1}: $\delta$ is a successor ordinal. 

In this case we trivialize with Goblot and extend to $\delta$ in an arbitrary way.

    \item \textbf{Case 2}: $\delta \in \acc(\lambda) \setminus S$ and $\delta > \sup(\acc(C_\delta))$. 
    
Then $\gamma \defeq \sup(\acc(C_\delta)) = \max(\acc(C_\delta))$ and $\mathrm{otp}(C_\delta \setminus \gamma)= \omega$ and $\mathrm{cf}(\delta)=\omega$. Thanks to Goblot, we find a trivialization $\Psi^\delta$. 

Now fix $\alpha< \delta$. If $\gamma = \alpha$, then let $\phi_{\alpha, \delta} =0$. Let $h= 0$ if $\alpha< \gamma$ and $\gamma < \alpha$. Define $\phi_{\alpha, \delta}: a_\alpha \cap a_\delta \rightarrow \mathbb{Z}$ as

\[ \phi_{\alpha, \delta} (x) = \begin{cases}
    (-1)^h \phi_{ \{\alpha, \gamma\} } (x) \text{ if } x \in a_\gamma \\

    -\psi_\alpha (x) \text{ if } x \not \in a_\gamma.
    
\end{cases}                       \]

First, we verify that $(\dagger)$ is preserved. For this, we need only to check that $\forall \beta \in \acc(C_\delta) \text{ } \forall \alpha < \beta \text{ } \phi_{\alpha, \beta} =^* \phi_{\alpha, \delta}$. If $\beta = \gamma$ then this is definitional. If $\beta \in \acc(C_\delta) \cap \gamma = \acc(C_\gamma)$ instead, then the induction hypothesis for $(\dagger)$ implies that $\phi_{\alpha, \beta} =^* \phi_{\alpha, \gamma}$. Since $a_\beta \subseteq^* a_\gamma \subseteq^* a_\delta$, this in turn implies that $\phi_{\alpha, \beta} =^* \phi_{\alpha, \delta}$, as desired. 

Next we verify coherence. We want to show that $\partial(\Phi)_{\alpha_0, \alpha_1, \delta} = \phi_{\alpha_1, \delta} - \phi_{\alpha_0, \delta} + \phi_{\alpha_0, \alpha_1} =^* 0$. If $\gamma \in \{ \alpha_0, \alpha_1\}$, then we only need to care about the domain $a_\gamma$, and our definitions yield either $\partial(\Phi)_{\alpha_0, \alpha_1, \delta} = \phi_{\alpha_1, \delta} - 0 - \phi_{\alpha_1, \delta}$ or $\partial(\Phi)_{\alpha_0, \alpha_1, \delta} = 0 - \phi_{\alpha_0, \delta}+ \phi_{\alpha_0, \delta}$. If $\gamma \not \in \{\alpha_0, \alpha_1 \}$, then, by the same definition, we have $\partial(\Phi)_{\alpha_0, \alpha_1, \delta} = \pm \partial(\Phi)_{\alpha_0, \alpha_1, \gamma} =^*0$ on $a_\gamma$ and $-\psi_{\alpha_1}-(-\psi_{\alpha_0}) + \phi_{\alpha_0, \alpha_1} =^* 0$ outside of $a_\gamma$.  

    \item \textbf{Case 3}: $\delta \in \acc(\lambda) \setminus S$ and $\sup(\acc(C_\delta))=\delta$.
    
For $\alpha< \delta$, let $\gamma(\alpha) = \min(\acc(C_\delta) \setminus (\alpha+1))$. Then, for all $x \in a_\alpha \cap a_\delta $, set 

\[  \phi_{\alpha, \delta} (x) = \begin{cases}
\phi_{\alpha, \gamma(\alpha)} (x) \text{ if } x \in a_{\gamma(\alpha)}
    \\
0 \text{            otherwise.}   
\end{cases}                           \]

Note that the ``otherwise'' case of the above definition only occurs for finitely many $x \in \dom(\phi_{\alpha, \delta})$. First we verify $(\dagger)$. For all $\beta \in \acc(C_\delta)$ and $\alpha < \beta$ we have $\phi_{\alpha, \delta} =^* \phi_{\alpha, \gamma(\alpha)} =^* \phi_{\alpha, \beta}$ where the second equality holds by the induction assumption since either $\gamma(\alpha)=\beta$ or $\gamma(\alpha)\in \acc(C_\beta)= \acc(C_\delta) \cap \beta$. Since moreover $a_\gamma(\alpha) \subseteq^* a_\beta \subseteq^* a_\delta$, we have $\phi_{\alpha, \delta} =^* \phi_{\alpha, \beta}$, as desired.

For coherence, let $\alpha_0< \alpha_1 < \delta$ and let $\gamma = \gamma(\alpha_1)$. Then as before $\partial(\Phi)_{\alpha_0, \alpha_1, \delta} = \partial(\Phi)_{\alpha_0, \alpha_1, \gamma} =^*0$.

    \item \textbf{Case 4}: $\delta \in S$ and not all of the conditions of Case $5$ below hold. 

Then there is nothing to check with respect to $(\dagger)$, and we proceed as in Case $3$.

    \item \textbf{Case 5}: $\delta \in S$ there exists some $g \in \mathcal{F}_\delta$ that reads as  $ T =\langle \tau_\xi: a_\xi \to \mathbb{Z} \mid \xi \in \dom(g) \rangle $ trivializing $\Phi \restriction \dom(g)$. 
    
     Note that there are at most $\vert \mathcal F_\delta \vert < 2^{\aleph_1}$ such $T$'s.  For each of these we choose an extension $T'$ trivializing $\Phi \restriction \delta $, which exists by Lemma \ref{lemma_extendtriv}. By Lemma \ref{cardinality_lemma}, there are $2^{\aleph_1}$-many equivalence classes of $1$-families with the same coboundary for the equivalence relation given by differing modulo a $1$-trivial family on the chain  $\mathcal{C} \restriction \delta$ of cofinality $\omega_1$, so we can find a trivialization $\Upsilon^\delta$ of $\Phi \restriction \delta$ such that $\Upsilon^\delta - T'$ is (coherent) nontrivial for every $g$ as above. Then let $\partial({\Upsilon^\delta}^\frown \langle 0 \rangle)= \Phi \restriction \delta+1$. Here again there is nothing to check with respect to $(\dagger)$.

\end{itemize}

Now take a putative trivialization $\Psi$ of $\Phi$, and assume that it is coded by some $f: \lambda \to \lambda = \frc$. Let $S' \subseteq S$ be the stationary set where the sequence $\langle \mathcal{F}_\delta \mid \delta \in S\rangle
$ guesses $f$ as promised. Then, for $\delta \in S'$, and for some $g \in \mathcal{F}_\delta$, we have that $ B =\{\beta \in \dom(g) \mid g(\beta) = f(\beta) \}$ is unbounded in $\delta$. Then $\partial( {\Upsilon^\delta}^\frown\langle0\rangle) = \Phi \restriction \delta+1 $, where $\Upsilon^\delta - T' = \Theta^\delta$ for some extension $T'$ of the family $T$ that was coded by $g$ and some coherent nontrivial $\Theta^\delta$. But then, for any extension $\Upsilon^*$ of $\Psi \restriction B$ to a trivialization of $\Phi \restriction \delta +1$, and for any $\nu \in B$ we have 

\[ \Psi_\nu -\Upsilon^*_\delta  =^* \Phi_{\delta, \nu}  =^* \Psi_\nu + \Theta^\delta_\nu.\]

Hence, $-\Upsilon^*_\delta$ would trivialize $\Theta^\delta \restriction B$, a contradiction since  $\Theta^\delta$ is nontrivial over $\mathcal{C}\restriction \delta$ and $\sup B = \delta$. 
\end{proof}

\begin{corollary}
    Suppose $\mathfrak{b}=\mathfrak{d}=\mathfrak{c} > \omega_1$, and $\clubsuit^{<2^{\aleph_1}}_\rightharpoonup(S)$ and $\square(\lambda,S)$  both hold for some stationary set $S \subseteq S^{\mathfrak{c}}_{\omega_1}$. Then $\lim^2 \mathbf{A} \neq 0$. 
\end{corollary}

\bibliographystyle{amsplain}
\bibliography{bibliography}

\providecommand{\bysame}{\leavevmode\hbox to3em{\hrulefill}\thinspace}
\providecommand{\MR}{\relax\ifhmode\unskip\space\fi MR }
\providecommand{\MRhref}[2]{%
  \href{http://www.ams.org/mathscinet-getitem?mr=#1}{#2}
}
\providecommand{\href}[2]{#2}
\begin{thebibliography}{10}

\bibitem{BANNISTER_2025}
Nathaniel Bannister, \emph{Nonvanishing higher derived limits without $w\diamondsuit _{\omega _1}$}, The Journal of Symbolic Logic (2025), 1–17.

\bibitem{bekkali2006topics}
Mohamed Bekkali, \emph{Topics in set theory: Lebesgue measurability, large cardinals, forcing axioms, rho-functions}, Springer, 2006.

\bibitem{berg}
Jeffrey Bergfalk, \emph{Strong homology, derived limits, and set theory}, Fund. Math. \textbf{236} (2017), no.~1, 71--82.

\bibitem{WITHOUTLC}
Jeffrey Bergfalk, Michael Hru\v{s}\'{a}k, and Chris Lambie-Hanson, \emph{Simultaneously vanishing higher derived limits without large cardinals}, Journal of Mathematical Logic \textbf{23} (2023), no.~01, 2250019.

\bibitem{CASAROSA2026111076}
Matteo Casarosa and Chris Lambie-Hanson, \emph{Simultaneously nonvanishing higher derived limits}, Advances in Mathematics \textbf{500} (2026), 111076.

\bibitem{dvzamonja1996saturated}
Mirna D{\v{z}}amonja and Saharon Shelah, \emph{Saturated filters at successors of singulars, weak reflection and yet another weak club principle}, Annals of Pure and Applied Logic \textbf{79} (1996), no.~3, 289--316.

\bibitem{farah1996coherent}
Ilijas Farah, \emph{A coherent family of partial functions on $\mathbb{N}$}, Proceedings of the American Mathematical Society \textbf{124} (1996), no.~9, 2845--2852.

\bibitem{Gob70}
R\'{e}mi Goblot, \emph{Sur les d\'{e}riv\'{e}s de certaines limites projectives. {A}pplications aux modules}, Bull. Sci. Math. (2) \textbf{94} (1970), 251--255. \MR{274557}

\bibitem{mardevsic1988strong}
S~Marde{\v{s}}i{\'c} and Andrei~V Prasolov, \emph{Strong homology is not additive}, Transactions of the American Mathematical Society \textbf{307} (1988), no.~2, 725--744.

\bibitem{mitchell1972rings}
Barry Mitchell, \emph{Rings with several objects}, Advances in Mathematics \textbf{8} (1972), no.~1, 1--161.

\bibitem{PRASOLOV2005493}
Andrei~V. Prasolov, \emph{Non-additivity of strong homology}, Topology and its Applications \textbf{153} (2005), no.~2, 493--527, Proceedings of the Second International Conference on Geometric Topology.

\bibitem{TALAYCO199569}
Daniel~E. Talayco, \emph{Applications of cohomology to set theory i: Hausdorff gaps}, Annals of Pure and Applied Logic \textbf{71} (1995), no.~1, 69--106.

\bibitem{VV24}
Boban Veli\v{c}kovi\'{c} and Alessandro Vignati, \emph{Non-vanishing higher derived limits}, Commun. Contemp. Math. \textbf{26} (2024), no.~7, Paper No. 2350031, 22. \MR{4760549}

\end{thebibliography}
\end{document}